\documentclass[reqno]{amsart}

\usepackage{mathrsfs}
\usepackage{amsmath}
\usepackage{amssymb}
\usepackage{cite}
\usepackage{latexsym}
\usepackage{graphicx}

\usepackage{color}
\usepackage{comment}

\theoremstyle{plain}
\newtheorem{thm}{Theorem}[section]
\newtheorem{lem}[thm]{Lemma}

\newtheorem{dfn}[thm]{Definition}
\newtheorem{prop}[thm]{Proposition}
\newtheorem{rmk}[thm]{Remark}

\def\C{\mathscr{C}}
\def\D{\mathrm{D}}
\def\G{\mathscr{G}}
\def\M{\mathscr{M}}
\def\N{\mathscr{N}}
\def\T{\mathrm{T}}

\def\d{\mathrm{d}}
\def\h{\mathrm{h}}
\def\r{\mathrm{r}}
\def\s{\mathrm{s}}
\def\u{\mathrm{u}}

\def\Cset{\mathbb{C}}
\def\Kset{\mathbb{K}}
\def\Nset{\mathbb{N}}
\def\Rset{\mathbb{R}}
\def\Sset{\mathbb{S}}

\def\id{\mathrm{id}}

\def\epsilon{\varepsilon}

\DeclareMathOperator{\sech}{sech}

\makeatletter
 \@addtoreset{equation}{section}
\makeatother
\def\theequation{\arabic{section}.\arabic{equation}}

\begin{document}


\title[Nonintegrability of perturbations of integrable system]%
{Nonintegrability of time-periodic perturbations
 of analytically integrable systems near homo- and heteroclinic orbits}

\author{Kazuyuki Yagasaki}

\address{Department of Applied Mathematics and Physics, Graduate School of Informatics,
Kyoto University, Yoshida-Honmachi, Sakyo-ku, Kyoto 606-8501, JAPAN}
\email{yagasaki@amp.i.kyoto-u.ac.jp}

\date{\today}

\begin{abstract}
We consider time-periodic perturbations of analytically integrable systems
 in the sense of Bogoyavlenskij
 and study their \emph{real-meromorphic} nonintegrability,
 using a generalized version due to Ayoul and Zung of the Morales-Ramis theory.
The perturbation terms are assumed to have finite Fourier series in time,
 and the perturbed systems are rewritten as higher-dimensional autonomous systems
 having the small parameter as a state variable.
We show that if the Melnikov functions are not constant,
 then the autonomous systems are not \emph{real-meromorphically} integrable
 near homo- and heteroclinic orbits.
We illustrate the theory for rotational motions of a periodically forced rigid body,
 which provides a mathematical model of a quadrotor helicopter.
\end{abstract}
\maketitle


\section{Introduction}

In this paper we study the nonintegrability of systems of the form
\begin{equation}
\dot{x}=f(x)+\epsilon g(x,\nu t),\quad
x\in\Rset^n,
\label{eqn:syse}
\end{equation}
where $\epsilon$ is a small parameter $0<\epsilon\ll 1$, $n\ge 3$ is an integer,
 $\nu>0$ is a constant,
 $f:\Rset^n\to\Rset^n$ and $g:\Rset^n\times\Rset\to\Rset^n$ are analytic
 and $g(x,\theta)$ is $2\pi$-periodic in $\theta$.
We adopt the following definition of integrability in the Bogoyavlenskij sense \cite{B98}.
 
\begin{dfn}[Bogoyavlenskij]
\label{dfn:1a}
For $n,q\in\Nset$ such that $1\le q\le n$,
 an $n$-dimensional dynamical system
\begin{equation}
\dot{x}=w(x),\quad z\in\Rset^n\text{ or }\Cset^n,
\label{eqn:gsys}
\end{equation}
is called \emph{$(q,n-q)$-integrable} or simply \emph{integrable} 
 if there exist $q$ vector fields $w_1(x)(:=w(x)),w_2(x),\dots,w_q(x)$
 and $n-q$ scalar-valued functions $F_1(x),\dots,\linebreak F_{n-q}(x)$ such that
 the following two conditions hold:
\begin{enumerate}
\setlength{\leftskip}{-1.8em}
\item[\rm(i)]
$w_1(x),\dots,w_q(x)$ are linearly independent almost everywhere $($a.e.$)$
 and commute with each other,
 i.e., $[w_j,w_k](x):=\D w_k(x)w_j(x)-\D w_j(x)w_k(x)\equiv 0$ for $j,k=1,\ldots,q$,
 where $[\cdot,\cdot]$ denotes the Lie bracket$;$
\item[\rm(ii)]
The derivatives $\D F_1(x),\dots, \D F_{n-q}(x)$ are linearly independent a.e.
 and $F_1(x),\linebreak\dots,F_{n-q}(x)$ are first integrals of $w_1, \dots,w_q$,
 i.e., $\D F_k(x)\cdot w_j(x)\equiv 0$ for $j=1,\ldots,q$ and $k=1,\ldots,n-q$,
 where ``$\cdot$'' represents the inner product.
\end{enumerate}
We say that the system \eqref{eqn:gsys}
 is \emph{analytically} $($resp. \emph{meromorphically}$)$ \emph{integrable}
 if the first integrals and commutative vector fields are analytic $($resp. meromorphic$)$. 
\end{dfn}
Definition~\ref{dfn:1a} is considered as a generalization of 
 Liouville-integrability for Hamiltonian systems \cite{A89,M99}
 since an $n$-degree-of-freedom Liouville-integrable Hamiltonian system
 with $n\ge 1$
 has not only $n$ functionally independent first integrals
 but also $n$ linearly independent commutative (Hamiltonian) vector fields
 generated by the first integrals.

When $\epsilon=0$, Eq.~\eqref{eqn:csys} becomes
\begin{equation}
\dot{x}=f(x).
\label{eqn:sys0}
\end{equation}
We make the following assumptions on the unperturbed system \eqref{eqn:sys0}:

\begin{figure}[t]
\includegraphics[scale=0.9]{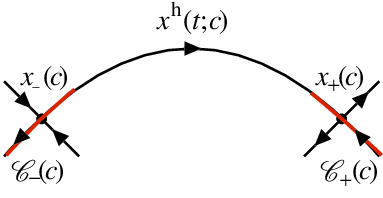}
\caption{Assumptions~(A2), (A4) and (A5).}
\label{fig:1a}
\end{figure}

\begin{enumerate}
\setlength{\leftskip}{-1.2em}
\item[(A1)]
There exist $m$ analytic first integrals $F_j:\Rset^n\to\Rset$, $j=1,\ldots,m$,
 where $1\le m\le n-2$.
\item[(A2)]
There exists a closed domain $I\subset\Rset^m$ with a nonempty interior such that
 for any $c\in I$ Eq.~\eqref{eqn:sys0}
 has two equilibria of nonhyperbolic saddle type, $x_\pm(c)$,
 at which $\D F_j(x)$, $j=1,\ldots,m$, are linearly independent
 on the level set $F(x)=(F_1(x),\ldots,F_m(x))=c$.
Moreover, $x_\pm(c)$ are analytic with respect to $c\in I^0=I\setminus\partial I$,
 and $\D F_1(x),\dots, \D F_m(x)$ are linearly independent near them.
\item[(A3)]
The system \eqref{eqn:syse} is analytically integrable, i.e.,
 it has $q$ commutative vector fields and $n-q$ first integrals
 satisfying conditions~(i) and (ii) in Definition~\ref{dfn:1a},
 where $n-q\ge m$ and $m$ of the first integrals are $F_j(x)$, $j=1,\ldots,m$.
Moreover, the derivatives of the additional $n-q-m$ first integrals
 are linearly dependent of $\D F_j(x)$, $j=1,\ldots,m$, at $x_\pm(c)$.
\end{enumerate}
Thus, the system \eqref{eqn:syse} represents a time-periodic perturbation
 of the analytic $(q,n-q)$-integrable system \eqref{eqn:sys0}.
\begin{enumerate}
\setlength{\leftskip}{-1.2em}
\item[(A4)]
$n_\s$ and $n_\u$ eigenvalues of $\D f(x_\pm(c))$ have negative and positive real parts,
 respectively, where $n_\s+n_\u+m=n$.
\end{enumerate}
By (A2), the zero eigenvalue of $\D f(x_\pm(c))$ is of geometric multiplicity $m$
 (see Section~1 of \cite{Y18}).
Moreover, by (A2)-(A4),
 the derivatives of the $n-q$ first integrals are not linearly independent at $x=x_\pm(c)$
 since the multiplicity of the zero eigenvalue of $\D f(x_\pm(c))$ is more than $m$ if not.

\begin{enumerate}
\setlength{\leftskip}{-1.2em}
\item[(A5)]
The two equilibria $x_\pm(c)$ are connected by a heteroclinic orbit $x^\h(t;c)$ satisfying
\[
\lim_{t\to\pm\infty}x^\h(t;c)=x_\pm(c)
\]
on each level set $F(x)=c$ for $c\in I$.
Moreover, $x^\h(t;c)$ is analytic in $c\in I^0$ as well as in $t\in\Rset$.
\item[(A6)]
For each $c\in I$
 there exists one-dimensional analytic, locally invariant manifolds with boundaries,
 $\C_\pm(c)\subset F^{-1}(c)$,
 such that $x_\pm(c)\in\C_\pm(c)\setminus\partial\C_\pm(c)$
 and $x^\h(t;c)\in\C_\pm(c)$ for $|t|\ge T$ with $T>0$ sufficiently large.
See Fig.~\ref{fig:1a}.
\end{enumerate}
Here ``local invariance'' of $\C_\pm(c)$ means
 that a trajectory $x(t)$ may stay in $\C_\pm(c)$ only for $t\in(t_1,t_2)$
 if it starts in $\C_\pm(c)$ at $t=0$, where $t_1<0<t_2$.
Thus, some trajectories starting in $\C_\pm(c)$
  may  escape $\C_\pm(c)$ through $\partial\C_\pm(c)$.
See \cite{E13,F72,F74,F77,W94} for details on this concept and its related matters.

The variational equation (VE) of \eqref{eqn:sys0}
 along the heteroclinic orbit $x^\h(t;c)$ is given by
\begin{equation}
\dot{\xi}=\D f(x^\h(t;c))\xi,\quad
\xi\in\Rset^n.
\label{eqn:ve}
\end{equation}
We easily see that $\xi=\dot{x}^\h(t;c)$ is a solution to \eqref{eqn:ve}
 which exponentially tends to $0$ as $t\to\pm\infty$.
By a fundamental result concerning asymptotic behavior of linear differential equations
 (e.g., Section~3.8 of \cite{CL55}),
 assumption~(A4) means that there exist $n_\s$ independent solutions to \eqref{eqn:ve}
 which exponentially tend to $0$ as $t\to\infty$
 and $n_\u$ independent solutions to \eqref{eqn:ve}
 which exponentially tend to $0$ as $t\to-\infty$ for each $c\in I$.
We assume the following.
\begin{enumerate}
\setlength{\leftskip}{-1.2em}
\item[(A7)]
The variational equation \eqref{eqn:ve} has no other independent solution
 than $\xi=\dot{x}^\h(t;c)$ such that it exponentially tends to $0$ as $t\to\pm\infty$.
\end{enumerate}
Heteroclinic transition motions occurring in systems satisfying (A1), (A2) and (A4)-(A7)
 were discussed in \cite{Y18} previously.
In (A2) and (A4)-(A7), it is allowed that $x_+(c)=x_-(c)$
 and $x^\h(t;c)$ becomes a homoclinic orbit to the equilibrium.

We rewrite \eqref{eqn:syse} as an autonomous system
\begin{equation}
\dot{x}=f(x)+\epsilon g(x,\theta),\quad
\dot{\theta}=\nu,\quad
(x,\theta)\in\Rset^n\times\Sset^1.
\label{eqn:asys}
\end{equation}
When $\epsilon=0$,
 there exist two $(m+1)$-dimensional normally hyperbolic invariant manifolds
 with boundaries,
\[
\M_0^\pm=\{(x_\pm(c),\theta)\mid\theta\in\Sset^1,c\in I\},
\]
which consist of nonhyperbolic saddles.
Here ``normal hyperbolicity'' means that
 the expansion and contraction rates of the flow
 generated by \eqref{eqn:asys} with $\epsilon=0$
 normal to $\M_0^\pm$ dominate those tangent to $\M_0^\pm$.
See \cite{E13,F72,F74,F77,W94} for the details on this concept and its related matters.
The invariant manifolds $\M_0^\pm$
 have $(m+n_\s+1)$- and $(m+n_\u+1)$-dimensional stable and unstable manifolds,
 $W^\s(\M_0^\pm)$ and $W^\u(\M_0^\pm)$,
 and $W^\s(\M_0^+)$ and $W^\u(\M_0^-)$
 intersect in the $(m+2)$-dimensional heteroclinic manifold
\[
\N=\{(x^\h(t;c),\theta)\mid\theta\in\Sset^1,t\in\Rset,c\in I\},
\]
i.e., $W^\s(\M_0^+)\cap W^\u(\M_0^-)\supset\N$.
When $\epsilon\neq 0$,
 there exist two $(m+1)$-dimensional normally hyperbolic,
 locally invariant manifolds with boundaries, $\M_\epsilon^\pm$, near $\M_0^\pm$,
 which have $(m+n_\s+1)$- and $(m+n_\u+1)$-dimensional stable and unstable manifolds,
 $W^\s(\M_\epsilon^\pm)$ and $W^\u(\M_\epsilon^\pm)$,
 near $W^\s(\M_0^\pm)$ and $W^\u(\M_0^\pm)$.
See Proposition~1 of \cite{Y18}.
  
We extend the domains of the independent and state variables, $t$ and $x$,
 in \eqref{eqn:syse} to those containing $\Rset$ and $\Rset^n$
 in $\Cset$ and $\Cset^n$, respectively.
The adjoint equation for the VE \eqref{eqn:ve} is given by 
\begin{equation}
\dot{\xi}=-\D f(x^\h(t;c))^\T\xi,
\label{eqn:ave}
\end{equation}
which we call the \emph{adjoint variational equation} (AVE)
 of \eqref{eqn:sys0} along $x^\h(t;c)$.
Let $\psi_2(x;c)$ be a bounded solution to the AVE \eqref{eqn:ave}
 such that $\psi_2(x;c)$ exponentially tends to zero as $t\to\pm\infty$.
It follows from assumption (A7) that $\psi_2(x;c)$ exists
 and Eq.~\eqref{eqn:ave} has no such solution
 that is linearly independent of $\psi_2(x;c)$ for $c\in I$ fixed
 (see Section~2.2).
We define the \emph{Melnikov function} as
\begin{equation}
M(\theta;c)=\int_{-\infty}^\infty\psi_2(t;c)\cdot g(x^\h(t;c),\nu t+\theta)\d t.
\label{eqn:M}
\end{equation}
In \cite{Y18} it was shown  that
 if $M(\theta;c)$ has a simple zero at $\theta=\theta_0$ for some $c\in I$,
 then for $|\epsilon|>0$ sufficiently small
 $W^\s(\M_\epsilon^+)$ and $W^\u(\M_\epsilon^-)$ intersect transversely.
Moreover, when the locally invariant manifolds $\M_\epsilon^\pm$
 consist of periodic orbits, heterocliinic transition motions
 resulting from transverse intersection
 between the $(n_\s+1)$- and $(n_\u+1)$-dimensional stable and unstable manifolds
 of the periodic orbits were described there.
 
We now state our main result.
We additionally assume the following:
\begin{enumerate}
\setlength{\leftskip}{-1em}
\item[(A8)]
The perturbation term $g(x,\theta)$ has a finite Fourier series, i.e.,
\[
g(x,\theta)
 =\sum_{j=-N}^N\hat{g}_j(x)e^{ij\theta},
\]
for some $N\in\Nset$.
\end{enumerate}
We see that
 the Fourier coefficients $\hat{g}_j(x)$, $j=-N,\ldots,N$, are analytic in $x$,
 since $g(x,\theta)$ is analytic in $(x,\theta)$.
In addition, we have $\hat{g}_j^\ast(x)=\hat{g}_{-j}(x)$
 since $g(x,\theta)$ must be real on $\Rset^n\times\Sset^1$, where
\[
\hat{g}_j(x)=\frac{1}{2\pi}\int_0^{2\pi}g(x,\theta)e^{-ij\theta}\d\theta,
\]
and the superscript `$\ast$' represents complex conjugate.
Under assumption (A8) we rewrite \eqref{eqn:syse} as
\begin{equation}
\begin{split}
&
\dot{x}=f(x)+\epsilon a_0(x)+\sum_{j=1}^N(a_j(x)u_j+b_j(x)v_j),\\
&
\dot{\epsilon}=0,\quad
\dot{u}_j=-j\nu v_j,\quad
\dot{v}_j=j\nu u_j,\quad
j=1,\ldots,N,
\end{split}
\label{eqn:rsys}
\end{equation}
where
\begin{align*}
&
a_0(x)=\hat{g}_0(x),\quad
a_j(x)=\hat{g}_j(x)+\hat{g}_{-j}(x),\\
&
b_j(x)=i(\hat{g}_j(x)-\hat{g}_{-j}(x)),\quad
j=1,\ldots,N.
\end{align*}
Note that $a_0(x)$, $a_j(x)$ and $b_j(x)$, $j=1,\ldots,N$, are real for $x\in\Rset^n$,
 and that $(u_j,v_j)=(\epsilon\cos j\nu t,\epsilon\sin j\nu t)$ is a solution
 to the $(u_j,v_j)$-components of \eqref{eqn:rsys} for $j=1,\ldots,N$.
Let $u=(u_1,\ldots,u_N)$ and $v=(v_1,\ldots,v_N)$.
Our main theorem is stated as follows.

\begin{thm}
\label{thm:main}
Suppose that the Melnikov function $M(\theta;c)$ is not constant
 under assumptions {\rm(A1)-(A8)}.
Then the system \eqref{eqn:rsys} is not real-meromorphically integrable near
\begin{align*}
\hat{\Gamma}(c)=&\{(x,\epsilon,u,v)=(x^\h(t;c),0,0,0)
 \in\Rset^n\times\Rset\times\Rset^N\times\Rset^N\mid t\in\Rset\}\\
& \cup\{(x_\pm(c),0,0,0)\}
\end{align*}
in $\Rset^{2N+n+1}$.
\end{thm}

In \cite{Y24b} the nonintegrability of time-periodic perturbations
 of single-degree-of-freedom Hamiltonian systems,
 where the perturbations are assumed to have finite Fourier series in time
 as in assumption~(A8),
 was discussed, based on a generalized version due to Ayoul and Zung \cite{AZ10}
 of the Morales-Ramis theory \cite{M99,MR01}.
See Appendix~A for its brief explanation
 and Section~2 of \cite{Y24b} for a little more details.
So it was shown that if the associated Melnikov function is not constant,
 then a system expanded as in \eqref{eqn:rsys} for \eqref{eqn:syse}
 is real-meromorphically nonintegrable near homo- and heteroclinic orbits.
The approaches of \cite{Y24b}
 are extended not in a straightforward way
 and applied to prove Theorem~\ref{thm:main}.

The author also developed a technique which permits us
 to prove \emph{complex-meromorphic} nonintegrability
 of nearly integrable dynamical systems
 near resonant periodic orbits in \cite{Y22,Y24c},
 based on the generalized version of the Morales-Ramis theory
 and its extension, the Morales-Ramis-Sim\'o theory \cite{MRS07}.
In particular, in \cite{Y22}, he showed for time-periodic perturbations
 of single-degree-of-freedom Hamiltonian systems
 that if a contour integral is not zero,
 then the perturbed system is not complex-meromorphically integrable
 near the periodic orbit
 such that the first integrals and commutative vector fields
 also depend complex-meromorphically on $\epsilon$ near $\epsilon=0$.
The technique was successfully applied to the Duffing oscillators in \cite{Y22}
 and to a forced pendulum in \cite{MY24}.
Another technique that is an extension of Poincar\'e \cite{P92} and Kozlov \cite{K83,K96}
 was developed for these systems in \cite{MY23}.
See also \cite{Y25} for a review of these techniques.
 
The outline of this paper is as follows:
In Section~2, we present preliminary results
 and reduce the problem to a lower dimensional system
 with a single frequency component.
In Section~3 we provide a proof of Theorem~\ref{thm:main}.
Finally,  in Section~4,
 we illustrate the theory for a periodically forced rigid body,
 which provides a mathematical model of a quadrotor helicopter.


\section{Preliminaries}

In this section we give preliminary results for the proof of Theorem~\ref{thm:main}.

\subsection{Reduction of the problem to single-frequency systems}
Letting $y_0=\epsilon$ and
\[
y_j=\tfrac{1}{2}(u_j+i v_j),\quad
y_{-j}=\tfrac{1}{2}(u_j-i v_j),\quad
j=1,\ldots,N,
\]
we rewrite \eqref{eqn:rsys} as
\begin{equation}
\dot{x}=f(x)+\sum_{j=-N}^N\hat{g}_j(x)y_j,\quad
\dot{y}_j=ij\nu y_j,\quad
j=-N,\ldots,N.
\label{eqn:csys}
\end{equation}
We easily see that
 $y_j=\tfrac{1}{2}\epsilon e^{ij\nu t}$ is a solution to the $y_j$-component of \eqref{eqn:csys}
 for $j\in\{-N,\ldots,N\}\setminus\{0\}$.
Let $y=(y_{-N},\ldots,y_N)$.

\begin{lem}
\label{lem:2a}
If the system \eqref{eqn:rsys} is real-meromorphically integrable
 near $\hat{\Gamma}$ in $\Rset^{2N+n+1}$,
 then the system \eqref{eqn:csys} is complex-meromorphically integrable near
\[
\tilde{\Gamma}(c)
 =\{(x,y)=(x^\h(t;c),0)\in\Rset^n\times\Rset^{2N+1}\mid t\in\Rset\}\cup\{(x_\pm(c),0)\}
\]
in $\Cset^{2N+n+1}$.
\end{lem}

\begin{proof}
The proof is similar to that of  Lemma~3.1 of \cite{Y24b}.
If $F(x,u,v,\epsilon)$ and $f(x,u,v,\epsilon)$ are, respectively,
 a first integral and commutative vector field for \eqref{eqn:rsys}
 near $\hat{\Gamma}$ in $\Rset^{2N+n+1}$,
 then so are they for \eqref{eqn:rsys} near $\hat{\Gamma}$ in $\Cset^{2N+n+1}$,
 and consequently so are $F(x,\tilde{u},\tilde{v},y_0)$ and $f(x,\tilde{u},\tilde{v},\epsilon)$
 for \eqref{eqn:csys} near $\tilde{\Gamma}$,
 where $\tilde{u}=(y_1+y_{-1},\ldots,y_N+y_{-N})$
 and $i\tilde{v}=(y_1-y_{-1},\ldots,y_N-y_{-N})$.
Thus, we obtain the desired result.
\end{proof}

The Melnikov function $M(\theta;c)$ is not constant if and only if
\begin{equation}
\hat{M}_j(c):=\int_{-\infty}^\infty f(x^\h(t;c))\cdot\hat{g}_j(x^\h(t;c))e^{ij\nu t}\d t\neq 0
\label{eqn:Mj}
\end{equation}
for some $j=\{-N,\ldots,N\}\setminus\{0\}$, since
\[
M(\theta;c)=\int_{-\infty}^\infty f(x^\h(t;c))\cdot
 \left(\sum_{j=-N}^N\hat{g}_j(x^\h(t;c))e^{ij(\nu t+\theta)}\right)\d t
 =\sum_{j=-N}^N\hat{M}_je^{ij\theta}.
\]
By Lemma~\ref{lem:2a}, to prove Theorem~\ref{thm:main},
 we only have to show that
 the system \eqref{eqn:csys} is complex-meromorphically integrable near $\tilde{\Gamma}(c)$
 if $\hat{M}_j(c)\neq 0$ for some $j\neq 0$.

We will apply Theorem~\ref{thm:MR}
 to the nonconstant particular solution $(x,y)=(x^\h(t;c),0)$ for \eqref{eqn:csys}
  in Section~3.
The VE of \eqref{eqn:csys} along the solution is given by
\begin{equation}
\dot{\xi}=\D f(x^\h(t;c))\xi+\sum_{j=-N}^N\hat{g}_j(x^\h(t;c))\eta_j,\quad
\dot{\eta}_j=ij\nu\eta_j,\quad
j=-N,\ldots,N.
\label{eqn:veh}
\end{equation}
Obviously, we have the following as in Lemma~3.2 of \cite{Y24b}.

\begin{lem}
\label{lem:2b}
If the identity component of the differential Galois group for \eqref{eqn:veh}
 is commutative,
 then so are those of its $(\xi,\eta_\ell)$-components with $\eta_j=0$, $j\neq\ell$,
\begin{equation}
\dot{\xi}=\D f(x^\h(t;c))\xi+\hat{g}_\ell(x^\h(t;c))\eta_\ell,\quad
\dot{\eta}_\ell=i\ell\nu\eta_\ell,
\label{eqn:vehl}
\end{equation}
for $\ell=-N,\ldots,N$.
\end{lem}

Based on Theorem~\ref{thm:MR} and Lemmas~\ref{lem:2a} and \ref{lem:2b},
 we only have to show that the identity component of the differential Galois group of \eqref{eqn:vehl} is not commutative
 if $\hat{M}_\ell(c)\neq 0$ for some $\ell\neq 0$,
 to prove Theorem~\ref{thm:main}.

\subsection{Reduction of the VE \eqref{eqn:vehl}}

Let $c=(c_1,\ldots,c_m)\in I$.
Substituting $x=x^\h(t;c)$ into \eqref{eqn:sys0}
 and differentiating the resulting equation with respect to $c_j$, $j=1,\ldots,m$,
 we easily see that
\begin{equation}
\xi=\frac{\partial x^\h}{\partial c_j}(t;c),\quad
j=1,\ldots,m,
\label{eqn:dxdc}
\end{equation}
are solutions to the variational equation \eqref{eqn:ve}.
We also differentiate the relations
\[
F(x^\h(t;c))=c,\quad
f(x_\pm(c))=0
\] 
with respect to $c_j$ to obtain
\begin{equation}
\D F(x^\h(t;c))\frac{\partial x^\h}{\partial c_j}(t;c)=e_j,\quad
\D f(x_\pm(c))\frac{\partial x_\pm}{\partial c_j}(c)=0,
\label{eqn:rel1}
\end{equation}
where $e_j\in\Rset^m$ is a vector of which the $j$th element is one and all the others are zero.
Hence, the solutions \eqref{eqn:dxdc} are linearly independent
 and tend to linearly independent eigenvectors for the zero eigenvalue of multiplicity $m$
 as $t\to\pm\infty$
 since $\lim_{t\to\pm\infty}\D F_j(x^\h(t;c))$, $j=1,\ldots,m$, are so by (A2).
 
Let $n_0=n_\s+n_\u$.
Using the above fact and Theorem~1 of \cite{G92},
 we immediately obtain the following lemma on the VE \eqref{eqn:ve}.
 
\begin{lem}
\label{lem:2c}
There exists  a fundamental matrix $\Phi(t;c)=(\phi_1(t;c),\ldots,\phi_n(t;c))$ to \eqref{eqn:ve}
 such that
\begin{equation}
\begin{split}
&
\lim_{t\to\pm\infty}\phi_1(t;c)=0,\quad
\lim_{t\to\pm\infty}\phi_2(t;c)=\infty,\\
&
\lim_{t\to+\infty}\phi_j(t;c)=0,\quad
\lim_{t\to-\infty}\phi_j(t;c)=\infty,\quad
3\le j\le n_\s+1,\\
&
\lim_{t\to+\infty}\phi_j(t;c)=\infty,\quad
\lim_{t\to-\infty}\phi_j(t;c)=0,\quad
n_\s+1<j\le n_0,\\
&
\lim_{t\to+\infty}|\phi_j(t;c)|<\infty,\quad
\lim_{t\to-\infty}|\phi_j(t;c)|<\infty,\quad
j>n_0,
\end{split}
\label{eqn:phi}
\end{equation}
where the convergence to zero and divergence for $\phi_j(t;c)$, $j\le n_0$,
 is exponentially fast.
\end{lem}
In particular, $\phi_1(t;c)$ is a bounded solution to the VE \eqref{eqn:ve},
 so that we can take $\phi_1(t;c)=\dot{x}^\h(t;c)$ by assumption~(A7).
We write
\[
\phi_1(t)=(\xi_{1\pm}+o(1))e^{\mp\lambda_{1\pm} t},\quad
\phi_2(t)=(\xi_{2\pm}+o(1))e^{\pm\lambda_{2\pm} t}
\]
as $t\to\pm\infty$,
 where $\xi_{j\pm}\in\Rset$, $j=1,2$, are nonzero constants,
 $\mp\lambda_{1\pm}$ and $\pm\lambda_{2\pm}$ are eigenvalues of $\D f(x_\pm(c))$
 and $\lambda_{j\pm}>0$, $j=1,2$.
We also see that
 $\Psi(t;c)=(\Phi(t;c)^{-1})^\T$ is a fundamental matrix to the AVE \eqref{eqn:ave}.
Let $\Psi(t;c)=(\psi_1(t;c),\ldots,\linebreak\psi_n(t;c))$.
We have
\begin{equation}
\begin{split}
&
\lim_{t\to\pm\infty}\psi_1(t;c)=\infty,\quad
\lim_{t\to\pm\infty}\psi_2(t;c)=0,\\
&
\lim_{t\to+\infty}\psi_j(t;c)=\infty,\quad
\lim_{t\to-\infty}\psi_j(t;c)=0,\quad
3\le j\le n_\s+1,\\
&
\lim_{t\to+\infty}\psi_j(t;c)=0,\quad
\lim_{t\to-\infty}\psi_j(t;c)=\infty,\quad
n_\s+1<j\le n_0,\\
&
\lim_{t\to+\infty}|\psi_j(t;c)|<\infty,\quad
\lim_{t\to-\infty}|\psi_j(t;c)|<\infty,\quad
j>n_0,
\end{split}
\label{eqn:psi}
\end{equation}
where the convergence to zero and divergence for $\psi_j(t;c)$, $j\le n_0$,
 are also exponentially fast.
In particular, $\psi_2(t;c)$ is the solution to the AVE \eqref{eqn:ave}
 in the definition \eqref{eqn:M} of the Melnikov function $M(\theta:c)$.
Moreover,
\[
\psi_1(t)=(\xi_{1\pm}^{-1}+o(1))e^{\pm\lambda_{1\pm} t},\quad
\psi_2(t)=(\xi_{2\pm}^{-1}+o(1))e^{\mp\lambda_{2\pm} t}
\]
as $t\to\pm\infty$,

\begin{rmk}
\label{rmk:2a}
From the argument at the beginning of this section,
 we take the solutions \eqref{eqn:dxdc} as $\phi_{n_0+j}(t;c)$, $j=1,\ldots,m$,
 in Lemma~$\ref{lem:2c}$.
Then $\psi_{n_0+j}(t;c)$, $j=1,\ldots,m$, in \eqref{eqn:psi} are given by
\begin{equation}
\xi=\D F_j(x^\h(t;c))^\T,\quad
j=1,\ldots,m.
\label{eqn:DF}
\end{equation}
Indeed, we differentiate
\[
\D F_j(x)\cdot f(x)=0
\]
with respect to $x$ to obtain
\[
\D^2 F_j(x)f(x)+\D F_j(x)\D f(x)=0,\quad
j=1,\ldots,m,
\]
so that Eq.~\eqref{eqn:DF} gives linearly independent solutions
 to the AVE \eqref{eqn:ave} since
\[
\frac{\d}{\d t}\D F_j(x^\h(t;c))=\D^2 F_j(x^\h(t;c))\dot{x}^\h(t;c)=\D^2 F_j(x^\h(t;c))f(x^\h(t;c)).
\]
Moreover, it follows from \eqref{eqn:rel1} that
 they tend to linearly independent eigenvectors of $\D f(x_\pm(c))$
 for the zero eigenvalue of multiplicity $m$ as $t\to\pm\infty$.
\end{rmk}

Letting $\tilde{\xi}\in\Cset$ and substituting $\xi=\tilde{\xi}\phi_2(t)$ into \eqref{eqn:vehl},
 we obtain
\begin{equation}
\dot{\tilde{\xi}}=\tilde{g}_\ell(x^\h(t;c))\eta_\ell,\quad
\dot{\eta}_\ell=i\ell\nu\eta_\ell,
\label{eqn:ve1}
\end{equation}
where
\[
\tilde{g}_\ell(x^\h(t;c))=\psi_2(t;c)\cdot\hat{g}_\ell(x^\h(t;c)),
\]
since $\phi_2(t)$ is a solution to \eqref{eqn:ve}.
It is clear that Lemma~\ref{lem:2b} can be modified as follows.

\begin{lem}
\label{lem:2d}
If the identity component of the differential Galois group for \eqref{eqn:veh}
 is commutative,
 then so are those of \eqref{eqn:ve1} for $\ell=-N,\ldots,N$.
\end{lem}

Taking the limits $t\to\pm\infty$ in \eqref{eqn:ve1}, we have
\begin{equation}
\dot{\tilde{\xi}}=0,\quad
\dot{\eta}_\ell=i\ell\nu\eta_\ell.
\label{eqn:ve10}
\end{equation}
We easily see that
\[
\tilde{\Phi}(t;c)=\begin{pmatrix}
1 & Y(t;c)\\
0 & e^{i\ell\nu t}
\end{pmatrix}\quad\mbox{and}\quad
\tilde{\Phi}_\pm(t;c)=\begin{pmatrix}
1 & 0\\
0 & e^{i\ell\nu t}
\end{pmatrix}
\]
are, respectively, fundamental matrices of \eqref{eqn:ve1} and \eqref{eqn:ve10}
  with $\tilde{\Phi}(0),\tilde{\Phi}_\pm(0)=\id_2$, where
\[
Y(t;c)=\int_0^t\tilde{g}_\ell(x^\h(\tau;c))e^{i\ell\nu\tau}\d\tau
=\int_0^t\psi_2(\tau;c)\cdot\hat{g}(x^\h(\tau;c))e^{i\ell\nu\tau}\d\tau.
\]
Thus, Eqs.~\eqref{eqn:ve1} and \eqref{eqn:ve10} can be solved completely.


\section{Proof of Theorem~\ref{thm:main}}

\begin{figure}
\includegraphics[scale=1]{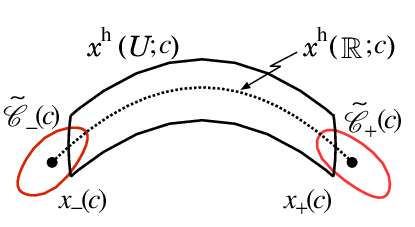}
\caption{Riemann surface $\Gamma=x^\h(U)\cup\tilde{\C}_+\cup\tilde{\C}_-$.}
\label{fig:4a}
\end{figure}

We are now in a position to prove Theorem~\ref{thm:main}.
Let $\ell\neq 0$ and fix $c\in I$.
Let $\tilde{\C}_\pm(c)$ be the complexification of $\C_\pm(c)$
 such that $\tilde{\C}_\pm(c)$ contain no other equilibrium than $x_\pm(c)$.
Let $R > 0$ be  sufficiently large and let $U$
 be a neighborhood of the open interval  $(-R,R) \subset \Rset$ in $\Cset$
 such that $x^\h(U;c)$ contains no equilibrium
 and intersects both $\C_+(c)$ and $\C_-(c)$.
Obviously, $x^\h(U;c)$ is a one-dimensional complex manifold with boundary.
We take $\Gamma=x^\h(U)\cup\tilde{\C}_+(c)\cup\C_-(c)$
 and the inclusion map as immersion $\iota:\Gamma\to\Cset^n$.
See Fig.~\ref{fig:4a}.
If $x_+(c)=x_-(c)$ and $x^\h(t;c)$ is a homoclinic orbit,
 then small modifications are needed in the definitions of $\Gamma$ and $\iota$.
 See Section~4.2 of \cite{BY12}.
Let $0_\pm\in\Gamma$ denote points corresponding to the equilibria $x_\pm(c)$.
Taking three charts, $\C_\pm(c)$ and $x^\h(U;c)$, 
 we rewrite the linear system \eqref{eqn:vehl} on $\Gamma$ as follows.

In $x^\h(U;c)$ we use the complex variable $t\in U$ as the coordinate
 and rewrite \eqref{eqn:ve1} as
\begin{equation}
\frac{\d\tilde{\xi}}{\d t}=\tilde{g}(\iota(t))\eta_\ell,\quad
\dot{\eta}_\ell=i\ell\nu\eta_\ell,
\label{eqn:ve2}
\end{equation}
which has no singularity there.
In $\C_+(c)$ and $\C_-(c)$
 there exist local coordinates $z_+$ and $z_-$, respectively,
 such that $z_\pm(0_\pm)=0$ and $\d/\d t=h_\pm(z_\pm)\d/\d z_\pm$,
 where $h_\pm(z_\pm)=\mp\lambda_{1\pm}z_\pm+O(|z_\pm|^2)$ are holomorphic functions
 near $z_\pm=0$.
We use the coordinates $z_\pm$ and rewrite \eqref{eqn:vehl} as
\begin{equation}
\frac{\d\tilde{\xi}}{\d z_\pm}
=\frac{\tilde{g}_\ell(z_\pm)}{h_\pm(z_\pm)}\eta_\ell,\quad
\frac{\d\eta_\ell}{\d z_\pm}=\frac{i\ell\nu}{h_\pm(z_\pm)}\eta_\ell,
\label{eqn:ve20}
\end{equation}
which have regular singularities at $z_\pm=0$.

\begin{lem}
\label{lem:3a}
If $\hat{M}_\ell(c)\neq 0$,
 then the identity component of the differential Galois group
 for the linear system consisting of \eqref{eqn:ve2} and \eqref{eqn:ve20} on $\Gamma$
 is not commutative.
\end{lem}

\begin{proof}
We take a point on $\Rset$ near $0_-$ as the base of the fundamental matrix
 and compute the monodromy matrices $M_\pm$
 for the linear system consisting of \eqref{eqn:ve2} and \eqref{eqn:ve20}
 around $z_\pm=0$ on $\Gamma$.
We easily obtain
\begin{equation*}
M_-=
\begin{pmatrix}
1 & 0\\
0 & e^{-2\pi\ell\nu/\lambda_{1-}}
\end{pmatrix}.
\end{equation*}
Analytic continuation of the fundamental matrix
 from $0_-$ to $0_+$ along $\Rset$ yields $\tilde{\Phi}(t;c)B_0$, where
\[
B_0=\lim_{t\to+\infty}\begin{pmatrix}
1 & Y(t;c)-Y(-t;c)\\[0.5ex]
0 & 1
\end{pmatrix}
=\begin{pmatrix}1 & \hat{M}_\ell(c)\\[0.5ex]
0 & 1
\end{pmatrix}.
\]
Hence, we have
\begin{equation*}
M_+=B_0^{-1}
\begin{pmatrix}
1 & 0\\
0 & e^{2\pi\ell\nu/\lambda_{1+}}
\end{pmatrix}B_0=\begin{pmatrix}
1 & (1-e^{2\pi\ell\nu/\lambda_{1+}})\hat{M}_\ell(c)\\
0 & e^{2\pi\ell\nu/\lambda_{1+}}.
\end{pmatrix}
\end{equation*}
We easily show that if $\hat{M}_\ell(c)\neq 0$,
 then $M_+$ and $M_-$ do not commute,
 so that the differential Galois group
 has a noncommutative identity component
 since it contains the monodromy group.
\end{proof}

\begin{proof}[Proof of Theorem~$\ref{thm:main}$]
Suppose that $M(\theta;c)$ is not a constant.
Then $\hat{M}_\ell(c)\neq 0$ for some $\ell\neq 0$, as stated in Section~2.1.
Hence, it follows from Lemma~\ref{lem:2d} that
 the identity component of the differential Galois group for \eqref{eqn:veh}
 is not commutative.
Using Theorem~\ref{thm:MR} and Lemma~\ref{lem:2a},
 we complete the proof.
\end{proof}

\begin{rmk}
\label{rmk:4a}
Our approach can apply to other time dependency of the perturbations.
See Remark~{\rm 4.2(i)} of {\rm\cite{Y24b}}.
\end{rmk}


\section{Example}

\begin{figure}[t]
\includegraphics[scale=1]{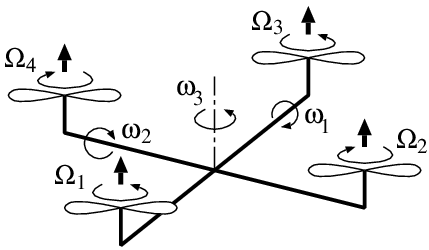}
\caption{Mathematical model for a quadrotor helicopter.}
\label{fig:4a}
\end{figure}

We apply our result to a periodically forced rigid body,
\begin{equation}
\begin{split}
\dot{\omega}_1
 =&\frac{I_2-I_3}{I_1}\omega_2\omega_3
 -\epsilon\left(\frac{\alpha}{I_1}\omega_2\sin\nu t+\delta_1\omega_1\right),\\
\dot{\omega}_2
 =&\frac{I_3-I_1}{I_2}\omega_3\omega_1
 +\epsilon\left(\frac{\alpha}{I_2}\omega_1\sin\nu t+\frac{\beta_2}{I_2}\sin\nu t
 -\delta_2\omega_2\right),\\
\dot{\omega}_3
 =&\frac{I_1-I_2}{I_3}\omega_1\omega_2
 +\epsilon\left(\frac{\beta_3}{I_3}\sin\nu t-\delta_3\omega_3\right),
\end{split}
\label{eqn:rbody}
\end{equation}
where $I_1,I_2,I_3,\alpha,\beta_2,\beta_3,\delta,\nu$ are nonnegative constants.
Equation~\eqref{eqn:rbody} represents a mathematical model of a quadrotor helicopter
 illustrated in Fig.~\ref{fig:4a} when one of the rotors becomes out of tune
 and its rotational speed is periodically modulated,
 where $\omega_j$ and $I_j$, $j=1,2,3$, respectively,
 represent the angular velocities and moments of inertia about the quadrotor's principal axes,
 and $\delta_j$, $j=1,2,3$, represent damping constants due to air resistance and so on.
See Section~5.1 of \cite{Y18} for the meaning of the other parameters.
Heteroclinic transition motions and persistence of first integrals and periodic orbits
 for \eqref{eqn:rbody} with $\delta_j=0$, $j=1,2,3$, were discussed in \cite{MY21,Y18}.
This example is also important from an application point of view
 since quadrotor helicopters are used or expected to be used in many areas.

Henceforth we assume that $I_1<I_2<I_3$.
When $\epsilon=0$, Eq.~\eqref{eqn:rbody} has two first integrals
\[
\tilde{F}_1(\omega)=\frac{1}{2}(I_1\omega_1^2+I_2\omega_2^2+I_3\omega_3^2),\quad
\tilde{F}_2(\omega)=I_1^2\omega_1^2+I_2^2\omega_2^2+I_3^2\omega_3^2,
\]
whose derivatives $\D\tilde{F}_1(\omega),\D\tilde{F}_2(\omega)$ are linearly independent a.e.,
 and it has two nonhyperbolic equilibria at $\omega=(0,\pm c,0)$
 connected by four heteroclinic orbits
\begin{align*}
&
\omega_\pm^\h(t;c)
 =\left(\pm c\sqrt{\frac{I_2(I_3-I_2)}{I_1(I_3-I_1)}}\sech kct,c\tanh kct,
 \pm c\sqrt{\frac{I_2(I_2-I_1)}{I_3(I_3-I_1)}}\sech kct\right),\\
&
\tilde{\omega}_\pm^\h(t;c)
 =\left(\pm c\sqrt{\frac{I_2(I_3-I_2)}{I_1(I_3-I_1)}}\sech kct,-c\tanh kct,
 \mp c\sqrt{\frac{I_2(I_2-I_1)}{I_3(I_3-I_1)}}\sech kct\right)
\end{align*}
on the level set $F_1(\omega):=\sqrt{2\tilde{F}_1(\omega)/I_2}=c>0$, where
\[
k=\sqrt{\frac{(I_2-I_1)(I_3-I_2)}{I_3I_1}}.
\]
See Fig.~\ref{fig:5a} for the unperturbed orbits of \eqref{eqn:rbody}
 with $\epsilon=0$ on the level set $F_1(\omega)=c$.
Note that $\D\tilde{F}_1(\omega)$ and $\D\tilde{F}_2(\omega)$ are not linearly independent
 at the equilibria $(0,\pm c,0)$.

\begin{figure}[t]
\includegraphics[scale=1.2]{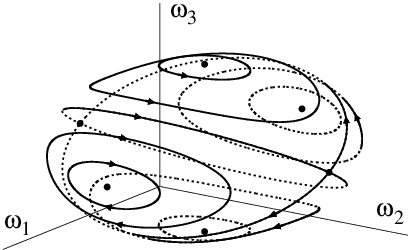}
\caption{Unperturbed orbits of \eqref{eqn:rbody} with $\epsilon=0$
 on the level set $F_1(\omega)=c$.}
\label{fig:5a}
\end{figure}

The variational equations of \eqref{eqn:rbody} with $\epsilon=0$
 along $\omega_\pm^\h(t;c)$ and $\tilde{\omega}_\pm^\h(t;c)$
 have no other independent solutions
 than $\xi=\dot{\omega}_\pm^\h(t;c)$ and $\dot{\tilde{\omega}}_\pm^\h(t;c)$, respectively,
 such that they tend to zero exponentially as $t\to\pm\infty$
 since the numbers of eigenvalues with positive and negative real parts are both one.
Thus, assumptions~(A1)-(A8) hold with $n=3$, $m,n_\s,n_\u,N=1$, $q=2$,
 and $I=[c_\ell,c_\r]$ for any $c_\ell,c_\r$ such that $0<c_\ell<c_\r<\infty$.
The corresponding adjoint variational equations
 have bounded solutions tending to zero as $t\to\pm\infty$,
\[
\psi_2(t;c)=(I_1(I_1-I_2)\omega_1(t;c),0,I_3(I_3-I_2)\omega_3(t;c))^\T,
\]
where $\omega(t;c)=\omega_\pm^\h(t;c)$ or $\tilde{\omega}_\pm^\h(t;c)$.
Moreover, we have
\[
\D F_1(\omega)=\frac{1}{cI_2}(I_1\omega_1,I_2\omega_2,I_3\omega_3)\neq 0
\]
when $F_1(\omega)=c$.

We compute the Melnikov functions as
\begin{align*}
M(\theta;c)
=&\int_{-\infty}^\infty[\alpha(I_2-I_1)\omega_1(t;c)\omega_2(t;c)
 +\beta_3(I_3-I_2)\omega_3(t;c)]\sin(\nu t+\theta_0)\d t\\
& -\int_{-\infty}^\infty[\delta_1 I_1(I_1-I_2)\omega_1(t;c)^2
 +\delta_3I_3(I_3-I_2)\omega_3(t;c)^2]\d t\\
=& \alpha(I_2-I_1)\cos\theta\int_{-\infty}^\infty\omega_1(t;c)\omega_2(t;c)\sin\nu t\,\d t\\
& +\beta_3(I_3-I_2)\sin\theta\int_{-\infty}^\infty\omega_3(t;c)\cos\nu t\,\d t\\
& -\delta_1 I_1(I_1-I_2)\int_{-\infty}^\infty\omega_1(t;c)^2\d t
  -\delta_3 I_3(I_3-I_2)\int_{-\infty}^\infty\omega_3(t;c)^2\d t,
\end{align*} 
so that
\begin{equation}
M_\pm(\theta;c)=\pm\alpha M_1(c)\cos\theta\pm\beta_3 M_2(c)\sin\theta
 -(\delta_3-\delta_1)M_3(c)
 \label{eqn:Mpm}
\end{equation}
for $\omega=\omega_\pm^\h$, and
\begin{equation}
\tilde{M}_\pm(\theta;c)=\pm\alpha M_1(c)\cos\theta\mp\beta_3 M_2(c)\sin\theta
 -(\delta_3-\delta_1)M_3(c)
\label{eqn:tMpm}
\end{equation}
for $\omega=\tilde{\omega}_\pm^\h$, where
\begin{align*}
&
M_1(c)=\pi\nu\sqrt{\frac{I_1I_2I_3^2}{(I_3-I_1)(I_3-I_2)}}\sech\left(\frac{\pi\nu}{2kc}\right),\\
&
M_2(c)=\pi\sqrt{\frac{I_1I_2(I_3-I_2)}{I_3-I_1}}\sech\left(\frac{\pi\nu}{2kc}\right),\quad
M_3(c)=\frac{2cI_2(I_2-I_1)(I_3-I_2)}{k(I_3-I_1)}.
\end{align*}

We rewrite \eqref{eqn:rbody} as
\begin{equation}
\begin{split}
&
\dot{\omega}_1
=\frac{I_2-I_3}{I_1}\omega_2\omega_3
 -\epsilon\delta_1\omega_1-\frac{\alpha}{I_1}\omega_2 v_1,\\
&
\dot{\omega}_2
=\frac{I_3-I_1}{I_2}\omega_3\omega_1
 -\epsilon\delta_2\omega_2
 +\frac{\alpha\omega_1+\beta_2}{I_2}v_1,\\
&
\dot{\omega}_3
 =\frac{I_1-I_2}{I_3}\omega_1\omega_2
 -\epsilon-\delta_3\omega_3+\frac{\beta_3}{I_3}v_1,\\
&
\dot{\epsilon}=0,\quad
\dot{u}_1=-\nu v_1,\quad
\dot{v}_1=\nu u_1.
\end{split}
\label{eqn:rbody1}
\end{equation}
Since $M_j(c)\neq 0$, $j=1,2$, we apply Theorem~\ref{thm:main}
 to obtain the following result for \eqref{eqn:rbody1}.
 
\begin{prop}
\label{prop:4a}
The system \eqref{eqn:rbody1} is not real meromorphically integrable near
\begin{align*}
&
\{(\omega,\epsilon,u_1,v_1)=(\omega_\pm^\h(t;c),0,0,0)
 \in\Rset^3\times\Rset\times\Rset\times\Rset\mid t\in\Rset\}\\
&
\cup\{(0,\pm,0,0,0,0)\in\Rset^6\}
\end{align*}
and 
\begin{align*}
&
\{(\omega,\epsilon,u_1,v_1)=(\tilde{\omega}_\pm^\h(t;c),0,0,0)
 \in\Rset^3\times\Rset\times\Rset\times\Rset\mid t\in\Rset\}\\
&
\cup\{(0,\pm,0,0,0,0)\in\Rset^6\}
\end{align*}
in $\Rset^6$.
\end{prop} 
 
\section*{Acknowledgements}
This work was partially supported by the JSPS KAKENHI Grant Number JP23K22409.


\appendix

\renewcommand{\theequation}{\Alph{section}.\arabic{equation}}

\section{Generalized Morales-Ramis theory}
We briefly review the Morales-Ramis theory for the general system \eqref{eqn:gsys}
 in a necessary setting.
See \cite{AZ10,M99,MR01} for more details on the theory.

Consider the general system \eqref{eqn:gsys}.
Let $z=\phi(t)$ be its nonconstant particular solution.
The VE of \eqref{eqn:gsys} along $z=\phi(t;c)$ is given by
\begin{equation}
\dot{\zeta} = \D w(\phi(t))\zeta,\quad \zeta \in \Cset^n.
\label{eqn:gve}
\end{equation}
Let $\C$ be a curve given by $z=\phi(t)$.
We take the meromorphic function field on $\C$
 as the coefficient field $\Kset$ of \eqref{eqn:gve}.
Using arguments given by Morales-Ruiz and Ramis \cite{M99,MR01}
 and Ayoul and Zung \cite{AZ10}, we have the following result.

\begin{thm}
\label{thm:MR}
Let $\G$ be the differential Galois group of \eqref{eqn:gve}.
If the system~\eqref{eqn:gsys} is meromorphically integrable near $\C$,
 then the identity component $\G^0$ of $\G$ is commutative.
\end{thm}

By contraposition of Theorem~\ref{thm:MR}, if $\G^0$ is not commutative,
 then the system~\eqref{eqn:gsys} is meromorphically nonintegrable near $\C$.


\end{document}